\documentclass[11pt, a4paper]{article}
\usepackage{float}
\usepackage[T1]{fontenc}
\usepackage{newtxtext}       % Times text font
\usepackage{newtxmath}       % Times math font
\usepackage[final]{microtype} % Professional kerning

\usepackage[left=1in, right=1in, top=1in, bottom=1in]{geometry}
\usepackage{authblk}
\makeatletter
\renewcommand{\maketitle}{
    \begin{center}
        {\LARGE \bfseries \@title \par}
        \vspace{1.0em}
        {\large \@author \par}
        \vspace{0.5em}
        {\itshape \@date}
    \end{center}
    \vspace{1.0em}
}
\makeatother

\usepackage{abstract}
\usepackage[authoryear, round]{natbib}
\usepackage{amsmath,amsthm,mathtools,bm}
\usepackage{aliascnt}

\usepackage[colorlinks=true, allcolors=blue]{hyperref}
\usepackage{xurl}
\usepackage{placeins}
\usepackage[capitalize,nameinlink]{cleveref}
\usepackage{siunitx}
\usepackage{booktabs}
\usepackage{graphicx}
\usepackage{listings}
\usepackage{algorithm}
\usepackage{algpseudocode}
\usepackage{orcidlink}

\newif\ifblind
\ifdefined\BLINDVERSION\blindtrue\else\blindfalse\fi

\theoremstyle{plain}
\newtheorem{theorem}{Theorem}

\newaliascnt{lemma}{theorem}

\aliascntresetthe{lemma}

\newaliascnt{proposition}{theorem}
\newtheorem{proposition}[proposition]{Proposition}
\aliascntresetthe{proposition}

\newaliascnt{corollary}{theorem}
\newtheorem{corollary}[corollary]{Corollary}
\aliascntresetthe{corollary}

\theoremstyle{definition}
\newaliascnt{definition}{theorem}
\newtheorem{definition}[definition]{Definition}
\aliascntresetthe{definition}

\theoremstyle{remark}
\newaliascnt{remark}{theorem}
\newtheorem{remark}[remark]{Remark}
\aliascntresetthe{remark}

\crefname{theorem}{Theorem}{Theorems}
\Crefname{theorem}{Theorem}{Theorems}
\crefname{lemma}{Lemma}{Lemmas}
\Crefname{lemma}{Lemma}{Lemmas}
\crefname{proposition}{Proposition}{Propositions}
\Crefname{proposition}{Proposition}{Propositions}
\crefname{corollary}{Corollary}{Corollaries}
\Crefname{corollary}{Corollary}{Corollaries}
\crefname{definition}{Definition}{Definitions}
\Crefname{definition}{Definition}{Definitions}
\crefname{remark}{Remark}{Remarks}
\Crefname{remark}{Remark}{Remarks}

\newcommand{\R}{\mathbb{R}}

\newcommand{\keywords}[1]{\vspace{1em}\noindent\textbf{Keywords:} #1}

\title{A Certifying MCKP Framework for $\Gamma$-Robust Discrete Pricing}
\ifblind
\author{Anonymous Author(s)}
\affil{}
\else
\author{Zi Yuan Eric Shao}
\affil{Department of Mathematics, ETH Zürich, \texttt{ershao@student.ethz.ch}}
\fi
\date{August 2026}

\begin{document}
\maketitle

\begin{abstract}
We study finite-menu portfolio pricing under a ratio margin requirement, price-admissibility constraints, and integer-budget demand uncertainty.  The pricing model reduces exactly to a multiple-choice knapsack problem (MCKP), while the coupled robust constraint is equivalent to a finite family of fixed-threshold MCKPs indexed by zero and every original option deviation.  For each threshold, classical upper-hull geometry yields an LP bound and a feasible one-item rounding rule whose loss is bounded by a single adjacent hull-value jump; under bounded jumps and linear objective growth, the corresponding relative bound is $O(1/n)$.  An exact-arithmetic branch-and-bound construction provides either a global optimum or a valid anytime gap over the complete threshold family, and direct evaluation of the original sorted-$\Gamma$ protection term checks every returned pricing decision.  The contribution is an end-to-end, pricing-specific certifying framework that connects these ingredients, states their boundary conditions, and implements the full original breakpoint family without tolerance clustering.
\end{abstract}


\section{Introduction and Reduction}\label{sec:intro}
We consider a pricing optimization problem where a decision maker must set prices for a collection of items in order to maximize overall value (e.g., profit or revenue), subject to two key constraints. First, the portfolio must achieve at least a prescribed margin level, so that overall profitability remains acceptable. Second, each individual price must stay sufficiently close to a given reference or ``fair'' price, reflecting regulatory, actuarial, or market considerations. Problems of this type arise repeatedly in practice, for instance in insurance tariff setting, retail and e-commerce pricing, or energy and telecom tariffs, whenever large price menus must be adjusted under margin and fairness constraints.

Two modeling choices deserve brief justification. We work with \emph{discrete} price menus because many applied settings (e.g., insurance tariffs, retail shelf prices with psychological endings, and standardized energy contracts) restrict the decision maker to a finite set of admissible price points, and continuous relaxations can therefore produce infeasible prices. We also impose the margin requirement as a \emph{ratio} constraint $N(\bm{x})/D(\bm{x}) \geq \Delta$ rather than an additive floor $N(\bm{x}) - D(\bm{x}) \geq b$. The ratio form is scale-invariant and matches standard profitability metrics such as loss ratio and gross-margin percentage.
\paragraph{Contributions.}
Our contributions are fourfold.
\begin{enumerate}
\item[(i)] We formulate a discrete portfolio pricing problem with ratio-based margin and fairness constraints and show that it reduces \emph{exactly} to a Multiple-Choice Knapsack Problem (MCKP). While the MCKP is classical, the reduction from constrained pricing with ratio margins, fairness bands, and weighted demand provides a direct bridge between pricing practice and combinatorial optimization theory.

\item[(ii)] We specialize the Bertsimas--Sim scalar-threshold representation to menu-dependent pricing deviations and prove that the resulting robust pricing problem is the maximum over a complete family of fixed-threshold MCKPs.  The candidate set contains zero and every original admissible option deviation (\Cref{thm:parametric,prop:theta_candidates_discrete_clean}).  The scalar representation itself is classical; the contribution here is the pricing specialization, the exact MCKP transformation of every member of the family, and an implementation that does not tolerance-cluster the original thresholds.

\item[(iii)] We turn the classical MCKP upper-hull relaxation into an explicit pricing certificate: round-down is robust feasible at the selected threshold and loses at most one adjacent hull-value jump (\Cref{prop:rg_additive_theta,prop:igap_additive_theta_clean}).  The relative gap is $O(1/n)$ under the uniform bounded-jump and linear-growth assumptions stated in \Cref{cor:igap_relative_theta_clean}.

\item[(iv)] We give a complete full-breakpoint branch-and-bound construction with valid global-gap accounting and a guarded parametric sweep (\Cref{sec:exact}).  The exactness theorems are exact-arithmetic statements.  The binary64 implementation preserves all distinct input thresholds, keeps below-hull integer options, and separately reports direct feasibility checks and numerical tolerances.
\end{enumerate}
The LP relaxation analysis in \Cref{sec:structure} consolidates known MCKP properties \citep{Dyer1984,Pisinger1995}; finite deviation enumeration is also established in the broader budgeted-robust literature \citep{BertsimasSim2004,Buesing2023}.  Accordingly, we do not claim either ingredient in isolation as new.  The distinct contribution is their pricing-specific composition, the one-item certificate stated in the transformed pricing quantities, and a reproducible full-family implementation with exact-search and failure-accounting safeguards.

Separate subsequent work takes the family established here as its starting point and studies simultaneous interval bounds for the complete threshold family.

\subsection{Problem Formulation}
\noindent Mathematically we have the following discrete problem. For each class $i\in\{1,\dots,n\}$, we choose one price level $x_i \in \mathcal X_i=\{x_{i,1},\dots,x_{i,m}\}$. We are given parameters for each class $i$: reference price $a_i>0$, strictly positive exposure $\omega_i>0$, strictly positive nominal demand $\hat{g}_i:\mathcal{X}_i\to \R_{>0}$ (monotone decreasing on the admissible range; e.g.\ piecewise-linear or iso-elastic), relative tolerance $\sigma_i$, and a margin target parameter $\Delta>0$.
We assume the demand functions $\hat{g}_i$ depend only on the own-price $x_i$ and not on other items' prices, which is appropriate when cross-product substitution is negligible (see \Cref{sec:conclusion} for extensions).

\noindent \textbf{Define}:
\[
N(\bm{x}) \coloneqq \sum_{i=1}^n \omega_i\,x_i\,\hat{g}_i(x_i),\qquad
D(\bm{x}) \coloneqq \sum_{i=1}^n \omega_i\,a_i\,\hat{g}_i(x_i)>0,\qquad
S(\bm{x}) \coloneqq N(\bm{x})-\Delta \cdot D(\bm{x})\ .
\]
\textbf{Optimization Problem}:
\begin{align}
\max_{\bm{x}}\quad & N(\bm{x})=\sum_{i=1}^n \omega_i x_i \hat{g}_i(x_i),\\
\text{s.\,t.}\quad &
\frac{N(\bm{x})}{D(\bm{x})}\ \ge\ \Delta\ \ \Leftrightarrow\ \ S(\bm{x})\ge0,\\
&
\frac{|x_i - a_i|}{a_i} \le \sigma_i \iff x_{i}\in [(1-\sigma_{i} )a_{i},\  (1+\sigma_{i} )a_{i}] \quad (\forall i).
\end{align}
The second constraint can be easily mitigated by defining the admissible menu
\[
I_{i}:=\left[ \left( 1-\sigma_{i} \right) a_{i},\  \left( 1+\sigma_{i} \right) a_{i} \right], \qquad
\mathcal X_i^{\mathrm{ad}} := \mathcal X_i \cap I_i \neq \emptyset,\qquad i=1,\dots,n,
\]
and henceforth restrict all decisions to $x_i\in \mathcal X_i^{\mathrm{ad}}$.

\subsection{Reduction to MCKP}\label{par:MCKP}
For a formal introduction to the classical 0--1 Knapsack problem and the standard MCKP definition, we refer to \Cref{app:background}.
Introduce binary selection variables $z_{ij}\in\{0,1\}$ indicating the choice $x_i=x_{i,j}$. For each option $x_{i,j}\in\mathcal X_i$, define
\begin{align}
    v_{ij} & \coloneqq \omega_i\,x_{i,j}\,\hat{g}_i(x_{i,j}), \\
    d_{ij} & \coloneqq \omega_i\,a_i\,\hat{g}_i(x_{i,j}), \\
    s_{ij} & \coloneqq v_{ij}-\Delta\,d_{ij}.
\end{align}
Let $J_i^{\mathrm{ad}}:=\{j\in\{1,\dots,m\}: x_{i,j}\in \mathcal X_i^{\mathrm{ad}}\}$ denote the admissible indices. Then
\[
N(\bm x)=\sum_{i=1}^n\sum_{j\in J_i^{\mathrm{ad}}} v_{ij}z_{ij},
\qquad
S(\bm x)=\sum_{i=1}^n\sum_{j\in J_i^{\mathrm{ad}}} s_{ij}z_{ij},
\qquad
\sum_{j\in J_i^{\mathrm{ad}}} z_{ij}=1 \ \ (\forall i).
\]
To obtain a single knapsack constraint with nonnegative costs, apply a baseline--slack transformation. For each $i$ choose a baseline option
\begin{align*}
    j^*(i) & \in \operatorname*{arg\,max}_{j\in J_i^{\mathrm{ad}}} s_{ij}, \\
    S_{\text{max}} & \coloneqq \sum_{i=1}^n s_{i\,j^*(i)}, \\
    c_{ij} & \coloneqq s_{i\,j^*(i)}-s_{ij}\ \ (\ge 0)\qquad (j\in J_i^{\mathrm{ad}}).
\end{align*}
Then
\[
\sum_{i=1}^n\sum_{j\in J_i^{\mathrm{ad}}} s_{ij}z_{ij}
\;=\;
\sum_{i=1}^n s_{i\,j^*(i)}
\;-\;
\sum_{i=1}^n\sum_{j\in J_i^{\mathrm{ad}}} c_{ij}z_{ij}
\;=\;
S_{\text{max}}-\sum_{i=1}^n\sum_{j\in J_i^{\mathrm{ad}}} c_{ij}z_{ij}.
\]
Hence the ratio constraint $S(\bm{x})\ge 0$ is equivalent to the knapsack inequality
\[
\sum_{i=1}^n\sum_{j\in J_i^{\mathrm{ad}}} c_{ij}\,z_{ij} \;\le\; S_{\text{max}}.
\]
The pricing problem is therefore exactly the multiple-choice knapsack problem (MCKP)
\begin{align}
    \max \quad & \sum_{i=1}^n \sum_{j\in J_i^{\mathrm{ad}}} v_{ij}\,z_{ij} \\
    \text{s.t.}\quad & \sum_{i=1}^n \sum_{j\in J_i^{\mathrm{ad}}} c_{ij}\,z_{ij} \;\le\; S_{\text{max}}, \\
    & \sum_{j\in J_i^{\mathrm{ad}}} z_{ij}=1 \quad (i=1,\dots,n), \\
    & z_{ij}\in\{0,1\}\quad (\forall i,\ j\in J_i^{\mathrm{ad}}).
\end{align}

\section{Related Work}
\label{sec:related_work}

\paragraph{MCKP foundations.}
The Multiple-Choice Knapsack Problem (MCKP) is a classical discrete optimization model in which exactly one option is selected from each of $n$
groups under a knapsack constraint. Early work formulated continuous and discrete variants and developed branch-and-bound, reduction, and
LP-relaxation methods; see, e.g., \citet{Nauss1978,Ibaraki1978,Sinha1979}. A key structural feature is the near-integrality of the LP relaxation:
there exists an optimal LP solution in which all but at most one group choose a single option, while one group (at most) is mixed between
two options \citep{Dyer1984}. This LP structure underpins many exact and heuristic schemes and is closely related to convex-hull and ``core''
ideas used in practical solvers \citep{Pisinger1995,Kellerer2004}.

\paragraph{Pricing and uncertainty.}
Discrete menus and business constraints are central in practical pricing models; for example, \citet{Harsha2019} formulate discrete
omnichannel prices with operational price and demand constraints, while
\ifblind recent work studies constrained logit-based pricing.\else
\citet{ShaoMaiCheng2026} study constrained logit-based pricing.\fi
These models capture substitution and richer business rules than the separable own-price model used here.  Robust pricing has also been
studied under resource uncertainty \citep{Thiele2009} and interval uncertainty in price sensitivity \citep{Hamzeei2022}.  Our question is
different: a finite own-price menu must meet one portfolio ratio-margin requirement when only an integer budget of item-level demand
contributions can deviate.

\paragraph{Budgeted robust discrete optimization.}
Robust knapsack algorithms are well studied in the single-choice setting \citep{Monaci2013}.  \citet{Caserta2019} instead consider a
multiple-choice multidimensional knapsack with ellipsoidal uncertainty and conic/linear reformulations.  For general robust binary
optimization with budget uncertainty, zero plus the original deviations is a known finite candidate family, and modern exact methods branch
over or reduce that family \citep{Buesing2023}.  These results prevent us from presenting scalar threshold enumeration as a new general
robust-optimization theorem.  What is less directly covered by the cited work, and what this paper supplies, is the complete chain for the
specific ratio-margin pricing model: menu-level uncertain contributions, fixed-threshold MCKP geometry, an interpretable one-item loss
certificate, direct sorted-$\Gamma$ validation, and global gap accounting in one public implementation.
% ============================================================
% SECTION: STRUCTURAL ANALYSIS
% ============================================================
\section{Robust Pricing under Demand Uncertainty}\label{sec:robust}

In the formulation presented in \Cref{sec:intro}, the demand response $\hat{g}_i(x_i)$ is treated as a deterministic function of price even though demand estimates are subject to statistical error. To mitigate this risk, we introduce a \textbf{Budgeted Robust} framework in which the true demand $\tilde{g}_i(x_i)$ resides in an uncertainty set $\mathcal U$ and the margin constraint must hold for \textit{all} realizations in that set. We then seek a price vector $\bm{x}$ that maximizes nominal revenue subject to this worst-case feasibility requirement.

\subsection{Uncertainty Model}
We assume the realized demand lies within a symmetric interval around the nominal value:
\begin{equation}
    \tilde{g}_i(x_i) \in [\hat{g}_i(x_i) - \delta_i(x_i), \hat{g}_i(x_i) + \delta_i(x_i)],
\end{equation}
where $\delta_i(x_i) \ge 0$ is a variability parameter (e.g., derived from the standard error of the demand estimation). We assume $\delta_i(x_i)<\hat g_i(x_i)$ on $\mathcal X_i^{\mathrm{ad}}$. Thus realized demand remains strictly positive under uncertainty and, because $\omega_i,a_i>0$, the realized ratio denominator $\widetilde D(\bm x)=\sum_i\omega_i a_i\widetilde g_i(x_i)$ is positive for every admissible price vector and every realization in the uncertainty set.

To avoid overly conservative solutions—such as those arising from box uncertainty where every product realizes its worst-case simultaneously—we adopt the ``Budget of Uncertainty'' approach proposed by \citet{BertsimasSim2004}. We define an integer budget $\Gamma \in \{0,1,\dots,n\}$, representing the maximum number of products that can simultaneously deviate to their worst-case limits.

\begin{definition}[Uncertainty Set]
Let $\xi_i \in [-1, 1]$ be the scaled deviation for item $i$. For a fixed price vector $x$, we define the uncertainty set $\mathcal{U}^\Gamma$ as:
\begin{equation}
    \mathcal{U}^\Gamma(\bm{x}) = \left\{ \tilde{g}(\bm{x}) \in \mathbb{R}^n \;\bigg|\; \tilde{g}_i(x_i) = \hat{g}_i(x_i) + \xi_i \delta_i(x_i), \quad \sum_{i=1}^n |\xi_i| \le \Gamma, \quad |\xi_i| \le 1 \quad \forall i \right\}.
\end{equation}
\end{definition}

\subsection{The Robust Counterpart}
The robust margin constraint requires $\widetilde N(\bm{x})/\widetilde D(\bm{x})\ge\Delta$ for all $\tilde{g} \in \mathcal{U}^\Gamma(\bm{x})$. The strict positivity of $\widetilde D$ established above makes this ratio condition equivalent to the following robust surplus inequality:
\begin{equation}\label{eq:robust_raw}
    \min_{\tilde{g} \in \mathcal{U}^\Gamma(\bm{x})} \sum_{i=1}^n \omega_i \tilde{g}_i(x_i) (x_i - \Delta a_i) \ge 0.
\end{equation}
Let $s_i(x_i)=\omega_{i} \left( x_{i}-\Delta a_{i} \right) \hat{g}_{i} \left( x_{i} \right)$ and $t_i(x_i):=\omega_i(x_i-\Delta a_i)\,\delta_i(x_i).$ The robust constraint can be rewritten as:
\begin{equation}\label{eq:robust_st}
    \sum_{i=1}^n s_i(x_i) + \min_{\xi : \sum |\xi_i| \le \Gamma, |\xi_i| \le 1} \sum_{i=1}^n t_i(x_i) \xi_i \ge 0.
\end{equation}
The minimization term represents the worst-case deviation permitted by the budget. Since the objective is linear in $\xi_i$, the minimum is achieved when $\xi_i$ takes values on the boundary of its domain to oppose the sign of $\omega_i(x_i-\Delta a_i)$. This worst-case behavior leads to the following exact formulation:

\begin{theorem}[Robust Counterpart Formulation]\label{thm:robust}
Let $\Gamma \in \{0,1,\dots,n\}$ be an integer. Define
\[
\beta(\bm{x},\Gamma)
\ :=\
\sum_{k=1}^{\Gamma} \ |t|_{(k)}(\bm{x}),
\]
where $|t|_{(1)}(\bm{x}) \ge \cdots \ge |t|_{(n)}(\bm{x})$ denotes the nonincreasing ordering of the numbers
$\big\{|t_i(x_i)|\big\}_{i=1}^n$. We adopt the convention that $\beta(\bm{x},0)=0$.
Then the robust optimization problem is equivalent to
\begin{align}
    \max_{\bm{x}} \quad & \sum_{i=1}^n \omega_i x_i \hat{g}_i(x_i) \\
    \text{s.t.} \quad
    & \sum_{i=1}^n s_i(x_i) - \beta(\bm{x}, \Gamma) \ge 0, \label{eq:robust_final}\\
    & x_i \in [(1-\sigma_i)a_i,(1+\sigma_i)a_i]\cap \mathcal X_i \qquad (\forall i). \nonumber
\end{align}
\end{theorem}

\begin{proof}
Fix $\bm{x}$ and consider the inner minimization in \eqref{eq:robust_st}:
\[
V(\bm{x})
=
\min_{\substack{\sum_{i=1}^n|\xi_i|\le \Gamma\\ |\xi_i|\le 1}}
\ \sum_{i=1}^n t_i(x_i)\,\xi_i .
\]
By feasibility and symmetry of the constraints, for any feasible $\bm{\xi}$ we may flip signs coordinatewise so that
$\operatorname{sgn}(\xi_i)=-\operatorname{sgn}(t_i(x_i))$ whenever $\xi_i\neq 0$, without violating
$\sum_i|\xi_i|\le \Gamma$ or $|\xi_i|\le 1$. This is valid because $|\xi_i|$ is invariant under sign changes, so neither the budget constraint
nor the box constraint is affected. Hence
\[
V(\bm{x})
=
-\max_{\substack{\sum_{i=1}^n y_i \le \Gamma\\ 0\le y_i \le 1}}
\ \sum_{i=1}^n |t_i(x_i)|\,y_i ,
\qquad\text{where } y_i:=|\xi_i|.
\]
The maximization problem above is a linear program. Since $\Gamma$ is an integer, it admits an optimal extreme point solution
with $y_i\in\{0,1\}$ for all $i$ and $\sum_i y_i\le \Gamma$.
Therefore an optimal solution sets $y_i=1$ for the $\Gamma$ indices with largest coefficients $|t_i(x_i)|$ and $y_i=0$ otherwise, giving
\[
V(\bm{x})
=
-\sum_{k=1}^{\Gamma} |t|_{(k)}(\bm{x})
=
-\beta(\bm{x},\Gamma).
\]
Substituting this into \eqref{eq:robust_st} yields \eqref{eq:robust_final}.
\end{proof}

\section{Structural Analysis and Greedy Optimality}\label{sec:structure}

The $\Gamma$-budget penalty $\beta(\bm{x},\Gamma)$ couples items and breaks separability, which is a main source of computational difficulty
in robust discrete pricing. Even without robustness, the reduced problem is NP-hard because it is an MCKP. The key structural observation is
that the LP relaxation remains efficiently solvable: each item's discrete menu induces an upper hull in the cost--value plane, and the global
LP reduces to greedy filling over hull segments. The sequel derives this hull-greedy LP structure and connects it to the robust model through
the parametric decomposition used in \Cref{sec:algorithm}.

\subsection{Upper-hull geometry of the LP relaxation}\label{subsec:hull_lp}
Fix an item $i$ and consider its admissible discrete options $x_{i,j}\in\mathcal X_i^{\mathrm{ad}}$ from \Cref{sec:intro}.
Each option induces a point $(c_{ij},v_{ij})$ in the cost--value plane, and we write
\[
\mathcal P_i \;:=\;\{(c_{ij},v_{ij}) : x_{i,j}\in\mathcal X_i^{\mathrm{ad}}\}\subset\mathbb R^2.
\]
In the LP relaxation, the constraints $\sum_j z_{ij}=1$ and $z_{ij}\ge 0$ allow fractional mixing across options; equivalently, item $i$
may realize any aggregate point $(c_i,v_i)\in\operatorname{conv}(\mathcal P_i)$. Since the objective maximizes value for a given resource
usage, only points on the \emph{upper hull} of $\operatorname{conv}(\mathcal P_i)$ can be optimal. This motivates
\emph{upper-hull filtering}: after removing cost--value dominated points, we replace $\mathcal P_i$ by the ordered vertex sequence
$\mathcal H_i$ of the efficient concave upper frontier of $\operatorname{conv}(\mathcal P_i)$, thereby discarding menu points without
changing the LP optimum.

\begin{remark}(Economic intuition)\label{rem:diminishing_returns}
In many pricing portfolios, spending additional margin slack yields progressively smaller revenue gains, suggesting a concave
``efficiency frontier''. The analysis below does not assume such concavity on the raw discrete menus: concavity arises automatically after
passing to the upper hull of $\operatorname{conv}(\mathcal P_i)$.
\end{remark}

Assume the LP is feasible, equivalently $C \ge \sum_i \min_j c_{ij}$ (and in particular $C \ge \sum_i c^H_{i,1}$).
\begin{remark}[Relation to prior work]\label{rem:prior_hull}
The properties collected in \Cref{thm:master_greedy} are individually standard in the MCKP literature; see \citet{Dyer1984}, \citet{Pisinger1995}, and \citet[Chapter~11]{Kellerer2004}. We re-prove them in unified form because our notation differs from these sources and the segment parameterization is used directly in \Cref{sec:algorithm}. Readers familiar with the classical MCKP LP structure may skip ahead to the Robust Constraint Structure subsection without loss of continuity.
\end{remark}
\begin{theorem}[Hull reduction and greedy optimality for the LP relaxation]\label{thm:master_greedy}
Consider the LP relaxation of the Multiple-Choice Knapsack Problem (MCKP):
\begin{align}\label{eq:lp_mckp_master}
\max_{\bm{z}} \quad & \sum_{i=1}^n\sum_{j=1}^m v_{ij}\,z_{ij}\\
\text{s.t.}\quad & \sum_{i=1}^n\sum_{j=1}^m c_{ij}\,z_{ij}\le C,\nonumber\\
& \sum_{j=1}^m z_{ij}=1 \quad (i=1,\dots,n),\nonumber\\
& z_{ij}\ge 0 \quad (\forall i,j).\nonumber
\end{align}
For each item $i$, define the point set $\mathcal P_i:=\{(c_{ij},v_{ij}) : j=1,\dots,m\}\subset\mathbb R^2$.
Remove cost--value dominated points from $\mathcal P_i$.  Let $\mathcal H_i=\{(c^H_{i,1},v^H_{i,1}),\allowbreak\dots,\allowbreak(c^H_{i,p_i},v^H_{i,p_i})\}$ denote the vertices of the resulting efficient concave upper frontier, ordered so that $c^H_{i,1}<\cdots<c^H_{i,p_i}$ and $v^H_{i,1}<\cdots<v^H_{i,p_i}$.
For $k=1,\dots,p_i-1$ define the hull segment lengths and slopes
\[
\Delta c^H_{i,k}:=c^H_{i,k+1}-c^H_{i,k}>0,
\qquad
\rho^H_{i,k}:=\frac{v^H_{i,k+1}-v^H_{i,k}}{c^H_{i,k+1}-c^H_{i,k}}.
\]

Then the following statements hold.

\begin{enumerate}
\item[(i)] (\textbf{Upper-hull sufficiency / two-adjacent-vertices mixing.})
There exists an optimal solution to \eqref{eq:lp_mckp_master} such that for each $i$,
the induced aggregate point $(c_i,v_i):=\big(\sum_j c_{ij}z_{ij},\sum_j v_{ij}z_{ij}\big)$ lies on the upper hull of
$\operatorname{conv}(\mathcal P_i)$, and hence can be represented as a convex combination of at most two \emph{adjacent} hull vertices
$(c^H_{i,k},v^H_{i,k})$ and $(c^H_{i,k+1},v^H_{i,k+1})$ for some $k$.

\item[(ii)] (\textbf{Reduction to continuous knapsack over hull segments.})
The LP \eqref{eq:lp_mckp_master} has the same optimal value as the continuous knapsack problem
\begin{align}\label{eq:cont_knap_master}
\max_{\bm{\ell}} \quad & \sum_{i=1}^n v^H_{i,1}+\sum_{i=1}^n\sum_{k=1}^{p_i-1}\rho^H_{i,k}\,\ell_{i,k}\\
\text{s.t.}\quad & \sum_{i=1}^n\sum_{k=1}^{p_i-1}\ell_{i,k}\ \le\ C-\sum_{i=1}^n c^H_{i,1},\nonumber\\
& 0\le \ell_{i,k}\le \Delta c^H_{i,k}\qquad(\forall i,k),\nonumber
\end{align}
where $\ell_{i,k}$ denotes the amount of capacity allocated to hull segment $k$ of item $i$.  Because slopes on each efficient
concave frontier are nonincreasing, a density-ordered optimum fills an item's segments in prefix order and therefore maps back to a valid
point on that item's frontier.  Arbitrary non-prefix feasible vectors of \eqref{eq:cont_knap_master} need not have such a mapping.

\item[(iii)] (\textbf{Greedy optimality for the LP relaxation.})
Let $\bm{\ell}^\star$ be obtained by filling segment variables $\ell_{i,k}$ in nonincreasing order of densities $\rho^H_{i,k}$,
each up to its upper bound $\Delta c^H_{i,k}$, with at most one final fractional fill to meet the capacity constraint.
Then $\bm{\ell}^\star$ induces an optimal solution to \eqref{eq:cont_knap_master} and hence an optimal solution to the LP relaxation
\eqref{eq:lp_mckp_master}. In particular, there exists an optimal LP solution in which at most one item is mixed fractionally
(i.e.\ across two adjacent hull vertices); all other items select a single hull vertex.

\item[(iv)] (\textbf{Special case: incremental greedy on the raw chain.})
Suppose for every $i$ the raw point chain $\{(c_{i,j},v_{i,j})\}_{j=1}^m$ satisfies
$c_{i,1}<\cdots<c_{i,m}$ and lie on the upper hull of $\operatorname{conv}(\mathcal P_i)$.
Thus they coincide with the hull vertices.
Then the greedy rule in (iii) reduces to the \emph{incremental greedy algorithm} on adjacent increments
$\rho_{i,j}=(v_{i,j}-v_{i,j-1})/(c_{i,j}-c_{i,j-1})$.
\end{enumerate}
\end{theorem}

\begin{proof}
Steps 1--3 (convex-hull reformulation, restriction to the upper hull, and segment-length parameterization) follow standard arguments adapted to our notation; we defer them to \Cref{app:thm2_steps} for completeness.

\textbf{Proof of (iii).}
Problem \eqref{eq:cont_knap_master} is a continuous knapsack (fractional) problem with items indexed by pairs $(i,k)$,
value density $\rho^H_{i,k}$ and capacity upper bound $\Delta c^H_{i,k}$.
A standard exchange argument shows that an optimal solution is obtained by filling variables in nonincreasing order of $\rho^H_{i,k}$:
take two variables $(i,k)$ and $(i',k')$ with $\rho^H_{i,k}>\rho^H_{i',k'}$. If $0<\ell_{i',k'}<\Delta c^H_{i',k'}$ while
$\ell_{i,k}<\Delta c^H_{i,k}$, then shifting an infinitesimal amount $\varepsilon$ of capacity from $\ell_{i',k'}$ to $\ell_{i,k}$
increases the objective by $\varepsilon(\rho^H_{i,k}-\rho^H_{i',k'})>0$ while preserving feasibility, contradicting optimality.
Thus, in any optimum, all higher-density variables are saturated before any lower-density variable is positive, implying the greedy fill rule
is optimal and that at most one variable can be fractional at termination. This proves (iii).

\textbf{Proof of (iv).}
If for each $i$ the raw chain coincides with the hull vertices, i.e.\
$\mathcal H_i=\{(c_{i,j},v_{i,j})\}_{j=1}^m$,
then the hull segments are exactly the adjacent increments $(j-1)\to j$ and $\rho^H_{i,k}$ equals $\rho_{i,j}$.
Therefore the greedy rule in (iii) reduces to selecting adjacent upgrades in descending order of $\rho_{i,j}$, which is precisely the
incremental greedy algorithm. \qedhere
\end{proof}

\subsection{Robust Constraint Structure}
\subsubsection{Box Uncertainty (\texorpdfstring{$\Gamma=n$}{Gamma=n}): Separability}
We first address the structural impact of robustness in the limiting case of \textbf{Box Uncertainty}, i.e.\ when the budget equals the number of items ($\Gamma=n$). In this regime, every item may simultaneously deviate to its worst-case limit.

Recall the robust margin constraint in the form \eqref{eq:robust_st}:
\[
\sum_{i=1}^n s_i(x_i) + \min_{\substack{\sum_{i=1}^n|\xi_i|\le \Gamma\\|\xi_i|\le 1}} \ \sum_{i=1}^n t_i(x_i)\,\xi_i \ \ge\ 0,
\qquad
s_i(x)=\omega_i(x-\Delta a_i)\hat g_i(x),\quad
t_i(x)=\omega_i(x-\Delta a_i)\delta_i(x).
\]
For $\Gamma=n$, the budget constraint $\sum_i|\xi_i|\le n$ is redundant given $|\xi_i|\le 1$, and the inner minimum is attained by choosing
$\xi_i=-\operatorname{sgn}(t_i(x_i))$, yielding
\begin{equation}\label{eq:box_constraint}
\sum_{i=1}^n s_i(x_i)\;-\;\sum_{i=1}^n |t_i(x_i)|\ \ge\ 0.
\end{equation}
Hence the box-robust pricing problem can be written as
\begin{align}
    \max_{\bm{x}} \quad & \sum_{i=1}^n v_i(x_i)
\qquad\text{where } v_i(x):=\omega_i\,x\,\hat g_i(x), \label{eq:box_problem}\\
    \text{s.t.} \quad & \sum_{i=1}^n\Big(s_i(x_i)-|t_i(x_i)|\Big)\ \ge\ 0, \nonumber\\
    & x_i \in \mathcal X_i\cap[(1-\sigma_i)a_i,(1+\sigma_i)a_i]\qquad(\forall i). \nonumber
\end{align}
In contrast to the coupled case $\Gamma<n$, the penalty in \eqref{eq:box_constraint} is additive and therefore separable across items.

\subsubsection{Coupled Uncertainty (\texorpdfstring{$\Gamma<n$}{Gamma<n}): Parametric Decomposition}
For the coupled case $\Gamma<n$, items interact through the shared uncertainty budget.
By \Cref{thm:robust}, the robust margin constraint is equivalent to
\begin{equation}\label{eq:robust_coupled_raw}
\sum_{i=1}^n s_i(x_i)\;-\;\beta(\bm{x},\Gamma)\ \ge\ 0,
\end{equation}
where $\beta(\bm{x},\Gamma)=\sum_{k=1}^\Gamma |t|_{(k)}(\bm{x})$ is the sum of the $\Gamma$ largest values of $\{|t_i(x_i)|\}_{i=1}^n$.
While this characterization is explicit, direct optimization over $\beta(\bm{x},\Gamma)$ is difficult because it couples all items
through a sorting operation. We resolve this by observing that $\beta(\bm{x},\Gamma)$ admits a dual scalar representation:
\begin{equation}\label{eq:inner_knap}
\beta(\bm{x},\Gamma)
=\max_{\bm{\xi}}\left\{\sum_{i=1}^n |t_i(x_i)|\,\xi_i \ \middle|\ \sum_{i=1}^n \xi_i\le \Gamma,\ 0\le \xi_i\le 1\ \forall i\right\}.
\end{equation}
\begin{theorem}[Parametric decomposition of the coupled $\Gamma$-budget penalty]\label{thm:parametric}
For every $\bm{x}$, the penalty $\beta(\bm{x},\Gamma)$ in \eqref{eq:inner_knap} admits the dual representation
\begin{equation}\label{eq:psi_dual}
\beta(\bm{x},\Gamma)
=\min_{\theta\ge 0}\Bigg\{\Gamma\theta+\sum_{i=1}^n \max\big\{0,\ |t_i(x_i)|-\theta\big\}\Bigg\}.
\end{equation}
Consequently, the robust margin constraint \eqref{eq:robust_coupled_raw} holds if and only if there exists $\theta\ge 0$ such that
\begin{equation}\label{eq:dual_robust_clean}
\sum_{i=1}^n\Big(s_i(x_i)-\max\big\{0,\ |t_i(x_i)|-\theta\big\}\Big)\ \ge\ \Gamma\theta.
\end{equation}
For any fixed $\theta\ge 0$, the left-hand side of \eqref{eq:dual_robust_clean} is separable across items.
\end{theorem}
\begin{proof}
Fix $\bm{x}$. Problem \eqref{eq:inner_knap} is a feasible and bounded linear program over a nonempty compact polytope.
Its dual is
\[
\min_{\theta\ge 0,\ \bm{\pi}\ge 0}\ \Gamma\theta+\sum_{i=1}^n \pi_i
\quad\text{s.t.}\quad
\theta+\pi_i\ge |t_i(x_i)|\ \ (\forall i).
\]
Strong duality yields equality of primal and dual optimal values. For any fixed $\theta\ge 0$, the minimal feasible choice is
$\pi_i=\max\{0,\ |t_i(x_i)|-\theta\}$, which gives \eqref{eq:psi_dual}. Substituting \eqref{eq:psi_dual} into
\eqref{eq:robust_coupled_raw} yields \eqref{eq:dual_robust_clean}.
\end{proof}
% ============================================================
% SECTION: A Robust Hull-Greedy Algorithm for the Discrete Problem
% ============================================================
\section{A Robust Hull-Greedy Algorithm for the Discrete Problem}\label{sec:algorithm}

We develop a method for the \emph{discrete} $\Gamma$-budget robust pricing problem. After the reduction in \Cref{sec:intro}, the
problem is an MCKP and remains NP-hard, so we target exact robust-feasibility certification together with explicit performance guarantees relative to LP benchmarks.
The algorithm is driven by two structural results: the one-dimensional parametric reformulation of the coupled robust constraint
(\Cref{thm:parametric}) and the exact hull-greedy solution of the fixed-parameter LP relaxation (\Cref{thm:master_greedy}).

\subsection{Exact robust feasibility via a scalar parameter}\label{subsec:theta}

For each admissible option $x_{i,j}\in\mathcal X_i^{\mathrm{ad}}$, define
\[
v_{ij}:=\omega_i\,x_{i,j}\,\hat g_i(x_{i,j}),\qquad
s_{ij}:=\omega_i\,(x_{i,j}-\Delta a_i)\,\hat g_i(x_{i,j}),\qquad
t_{ij}:=\omega_i\,(x_{i,j}-\Delta a_i)\,\delta_i(x_{i,j}).
\]
By \Cref{thm:parametric}, the robust margin constraint is equivalent to the existence of $\theta\ge 0$ such that
\begin{equation}\label{eq:theta_feas_clean}
\sum_{i=1}^n\Big(s_i(x_i)-\max\{0,\ |t_i(x_i)|-\theta\}\Big)\ \ge\ \Gamma\theta.
\end{equation}
For fixed $\theta\ge 0$, introduce the per-option $\theta$-modified contributions
\begin{equation}\label{eq:s_theta_def_clean}
s^\theta_{ij}\;:=\;s_{ij}-\max\{0,\ |t_{ij}|-\theta\}.
\end{equation}
Then \eqref{eq:theta_feas_clean} becomes the separable inequality
\begin{equation}\label{eq:theta_margin_clean}
\sum_{i=1}^n s^\theta_{i,\,j(i)}\ \ge\ \Gamma\theta,
\qquad j(i)\in\{1,\dots,m\}\ \text{chosen for each }i.
\end{equation}
\noindent where $j(i)$ denotes the index of the option selected for item $i$.

\begin{proposition}[Finite candidate set for \texorpdfstring{$\theta$}{theta}]\label{prop:theta_candidates_discrete_clean}
Let
\[
\mathcal B\;:=\;\{0\}\ \cup\ \{|t_{ij}|:\ i=1,\dots,n,\ x_{i,j}\in\mathcal X_i^{\mathrm{ad}}\}.
\]
If the discrete robust problem is feasible, then there exists an optimal solution $(\bm x^\star,\theta^\star)$ with $\theta^\star\in\mathcal B$.
\end{proposition}

\begin{proof}
Fix a discrete $\bm x$. Define
\[
f_{\bm x}(\theta)
:=\sum_{i=1}^n\Big(s_i(x_i)-\max\{0,\ |t_i(x_i)|-\theta\}\Big)-\Gamma\theta.
\]
Then $f_{\bm x}$ is concave piecewise-affine in $\theta$, with breakpoints contained in $\{0\}\cup\{|t_i(x_i)|\}_{i=1}^n$.
If $f_{\bm x}(\theta_0)\ge 0$ for some $\theta_0\ge 0$ and $\theta_0$ lies in a bounded affine interval, the maximum of that affine
function over the interval occurs at an endpoint, so at least one endpoint is also nonnegative.  On the final ray
$[\max_i|t_i(x_i)|,\infty)$, the function equals $\sum_i s_i(x_i)-\Gamma\theta$: it is constant when $\Gamma=0$ and nonincreasing when
$\Gamma>0$, so feasibility anywhere on the ray implies feasibility at its finite left endpoint.  Thus feasibility for fixed $\bm x$ can
always be certified by some breakpoint $\theta\in\{0\}\cup\{|t_i(x_i)|\}_{i=1}^n$. Because each $|t_i(x_i)|$ equals $|t_{i,j(i)}|$ for some admissible option,
every such value belongs to~$\mathcal B$. Since this argument holds for every discrete $\bm x$ individually, the joint optimum
$(\bm x^\star,\theta^\star)$ is covered by enumerating $\theta\in\mathcal B$ and, for each fixed~$\theta$, optimizing over discrete~$\bm x$---which is precisely the structure exploited by Algorithm~\ref{alg:robust_hull_greedy}.
\end{proof}

\subsection{Fixed-\texorpdfstring{$\theta$}{theta} reformulation as an MCKP}\label{subsec:theta_mckp}

Fix $\theta\in\mathcal B$. Convert \eqref{eq:theta_margin_clean} into a knapsack inequality with nonnegative costs via a baseline--slack
transformation. For each item $i$, choose
\[
j^\theta(i)\in\operatorname*{arg\,max}_{j}\ s^\theta_{ij},\qquad
S^\theta_{\max}:=\sum_{i=1}^n s^\theta_{i,\,j^\theta(i)},\qquad
c^\theta_{ij}:=s^\theta_{i,\,j^\theta(i)}-s^\theta_{ij}\ \ (\ge 0).
\]
Then \eqref{eq:theta_margin_clean} is equivalent to
\begin{equation}\label{eq:theta_knapsack_clean}
\sum_{i=1}^n c^\theta_{i,\,j(i)}\ \le\ C^\theta,
\qquad
C^\theta:=S^\theta_{\max}-\Gamma\theta.
\end{equation}
Thus, for fixed $\theta$, robust-feasible discrete pricing is exactly the MCKP
\begin{equation}\label{eq:mckp_theta_clean}
\max \ \sum_{i=1}^n v_{i,\,j(i)}
\quad\text{s.t.}\quad
\sum_{i=1}^n c^\theta_{i,\,j(i)}\le C^\theta,
\quad
x_{i,j(i)}\in\mathcal X_i^{\mathrm{ad}}\ \ (\forall i).
\end{equation}

\subsection{Upper-hull filtering and an exact LP solver for fixed \texorpdfstring{$\theta$}{theta}}\label{subsec:hull_theta}

Fix $\theta\in\mathcal B$ with $C^\theta\ge 0$. Item $i$ induces the point set
\[
\mathcal P_i^\theta:=\{(c^\theta_{ij},v_{ij}) : x_{i,j}\in\mathcal X_i^{\mathrm{ad}}\}\subset\mathbb R^2.
\]
Before computing the hull, apply a standard preprocessing step: (a) if multiple points share the same cost $c$, keep only the one with
largest value $v$; (b) remove dominated points (if $c\le c'$ and $v\ge v'$, discard $(c',v')$). This does not change the upper hull.
Apply \emph{upper-hull filtering} by replacing $\mathcal P_i^\theta$ with the ordered vertex sequence
\[
\mathcal H_i^\theta=\{(c^{H,\theta}_{i,1},v^{H,\theta}_{i,1}),\dots,(c^{H,\theta}_{i,p_i},v^{H,\theta}_{i,p_i})\}
\]
of the upper hull of $\operatorname{conv}(\mathcal P_i^\theta)$ (ordered by increasing $c$). Since $\mathcal H_i^\theta$ consists of
vertices of $\operatorname{conv}(\mathcal P_i^\theta)$, each hull vertex is one of the original menu points $(c^\theta_{ij},v_{ij})$ and
therefore corresponds to a discrete option.
Define hull segment lengths and slopes
\[
\Delta c^{H,\theta}_{i,k}:=c^{H,\theta}_{i,k+1}-c^{H,\theta}_{i,k}>0,\qquad
\rho^{H,\theta}_{i,k}:=\frac{v^{H,\theta}_{i,k+1}-v^{H,\theta}_{i,k}}{\Delta c^{H,\theta}_{i,k}}
\quad (k=1,\dots,p_i-1).
\]
By \Cref{thm:master_greedy}, the LP relaxation of \eqref{eq:mckp_theta_clean} is solved \emph{exactly} by greedily filling hull segments in
nonincreasing order of $\rho^{H,\theta}_{i,k}$. Moreover, there exists an optimal LP solution in which at most one item is fractional (mixed
between two adjacent hull vertices).

\paragraph{Exact robust-feasibility certificate (final check).}\label{par:cert}
For any discrete price vector $\bm x$, robust feasibility can be certified directly by
\begin{equation}\label{eq:certify_alg}
\sum_{i=1}^n s_i(x_i)\;-\;\beta(\bm{x},\Gamma)\ \ge\ 0,
\qquad
\beta(\bm{x},\Gamma)=\sum_{k=1}^{\Gamma}|t|_{(k)}(\bm{x}),
\end{equation}
where $t_i(x)=\omega_i(x-\Delta a_i)\delta_i(x)$ and $|t|_{(1)}(\bm{x})\ge\cdots\ge|t|_{(n)}(\bm{x})$ are the sorted magnitudes.

\subsection{Discrete recovery: one-item rounding with optional repair}\label{subsec:rounding}

Let $\bm z^{\mathrm{LP}}(\theta)$ be an optimal LP solution produced by the hull-greedy procedure in \Cref{subsec:hull_theta}.
If no item is fractional, it is already discrete. Otherwise, exactly one item $i^\star$ is mixed between two adjacent hull vertices
$(c^{H,\theta}_{i^\star,k},v^{H,\theta}_{i^\star,k})$ and $(c^{H,\theta}_{i^\star,k+1},v^{H,\theta}_{i^\star,k+1})$.

\paragraph{Default choice (feasibility-preserving round-down).}
Rounding $i^\star$ \emph{down} to the adjacent vertex with smaller knapsack cost $c^\theta$ weakly decreases total cost and thus preserves
\eqref{eq:theta_knapsack_clean}; all other items remain unchanged. This yields a feasible discrete solution $\bm x^{\downarrow}(\theta)$.

\paragraph{Optional improvement (round-up + repair).}
In practice, rounding down can lose a larger single-item value jump than necessary. As an optional refinement, also consider rounding
$i^\star$ \emph{up} to the higher vertex, obtaining $\bm x^{\uparrow}(\theta)$, which may violate \eqref{eq:theta_knapsack_clean} by some
excess cost $\Delta C>0$. If $\Delta C>0$, attempt to restore feasibility by a \emph{repair} step that performs discrete \emph{downgrades} on other
items along their hull chains: repeatedly select a downgrade that reduces knapsack cost by at least a small amount while minimizing loss of
objective, specifically by selecting the downgrade $(i, k \to k{-}1)$ along item $i$'s hull chain that minimizes the ratio
$(v^{H,\theta}_{i,k}-v^{H,\theta}_{i,k-1})/(c^{H,\theta}_{i,k}-c^{H,\theta}_{i,k-1})$ among items with remaining downgrade capacity.  The
process terminates after finitely many steps; if the available downgrades do not remove the excess, the round-up candidate is discarded.
Finally, keep the better of the feasible candidates,
\[
\bm x(\theta)\in\arg\max\{N(\bm x^{\downarrow}(\theta)),\ N(\bm x^{\uparrow\!\to\!\mathrm{repair}}(\theta))\}.
\]
This refinement is heuristic but feasibility-preserving. All theoretical guarantees (Propositions~\ref{prop:rg_additive_theta} and~\ref{prop:igap_additive_theta_clean}) apply to the round-down solution $\bm x^{\downarrow}(\theta)$ alone; the optional round-up-and-repair and completion steps can only improve the objective and are not required for the stated bounds.

\paragraph{Feasibility-preserving completion.}
While capacity remains, apply discrete \emph{upgrades} along each item's hull chain (adjacent hull vertices) in decreasing order of slope
among upgrades that fit. This can only improve the objective and preserves \eqref{eq:theta_knapsack_clean}.

\subsection{Complete algorithm and guarantees}\label{subsec:rg_full}

We enumerate $\theta\in\mathcal B$ (which is sufficient by \Cref{prop:theta_candidates_discrete_clean}), solve the corresponding LP
relaxation exactly by hull-greedy, and recover a discrete feasible solution by one-item rounding (optionally with repair/completion); see
Algorithm~\ref{alg:robust_hull_greedy}.

\begin{proposition}[Additive performance bound (per fixed $\theta$)]\label{prop:rg_additive_theta}
Fix $\theta\in\mathcal B$ with $C^\theta\ge 0$. Let $\mathrm{OPT}_{\mathrm{LP}}(\theta)$ and $\mathrm{OPT}_{\mathrm{IP}}(\theta)$ denote the
optimal objective values of the LP relaxation of \eqref{eq:mckp_theta_clean} and of \eqref{eq:mckp_theta_clean}, respectively.
Let $\bm x(\theta)$ be the rounded solution produced by Algorithm~\ref{alg:robust_hull_greedy} (before optional post-processing). Then
\[
\mathrm{OPT}_{\mathrm{LP}}(\theta)-N(\bm x(\theta))\ \le\ \Delta V_{\max}^\theta,
\qquad\text{and hence}\qquad
\mathrm{OPT}_{\mathrm{IP}}(\theta)-N(\bm x(\theta))\ \le\ \Delta V_{\max}^\theta,
\]
where $\Delta V_{\max}^\theta$ is defined in \Cref{subsec:igap}.
\end{proposition}

\begin{proof}
By \Cref{thm:master_greedy}(iii), there exists an optimal LP solution mixing at most one item between two adjacent hull vertices.
Rounding down that item (to the adjacent vertex with smaller $c^\theta$) preserves feasibility and decreases the objective by at most the
corresponding adjacent value jump, bounded by $\Delta V_{\max}^\theta$. The second inequality follows from
$\mathrm{OPT}_{\mathrm{LP}}(\theta)\ge \mathrm{OPT}_{\mathrm{IP}}(\theta)$.
\end{proof}

\begin{algorithm}[H]
\footnotesize
\caption{Robust hull-greedy with $\theta$-enumeration and discrete recovery}\label{alg:robust_hull_greedy}
\begin{algorithmic}[1]
\State \textbf{Input:} admissible menus $\{\mathcal X_i^{\mathrm{ad}}\}$, data $(v_{ij},s_{ij},t_{ij})$, budget $\Gamma$.
\State Form $\mathcal B=\{0\}\cup\{|t_{ij}|:\ x_{i,j}\in\mathcal X_i^{\mathrm{ad}}\}$, deduplicate, and sort increasingly.
\State $\textsc{BestVal}\leftarrow -\infty$, $\textsc{BestSol}\leftarrow \emptyset$.
\For{$\theta\in\mathcal B$}
    \State Compute $s^\theta_{ij}$ via \eqref{eq:s_theta_def_clean}.
    \State Choose $j^\theta(i)\in\arg\max_j s^\theta_{ij}$ and set $c^\theta_{ij}=s^\theta_{i,\,j^\theta(i)}-s^\theta_{ij}$.
    \State $C^\theta\leftarrow \sum_{i=1}^n s^\theta_{i\,j^\theta(i)}-\Gamma\theta$. \If{$C^\theta<0$} \textbf{continue}. \EndIf
    \For{$i=1,\dots,n$}
        \State Build $\mathcal P_i^\theta=\{(c^\theta_{ij},v_{ij})\}_j$.
        \State Preprocess $\mathcal P_i^\theta$: merge equal costs (keep max $v$), then remove dominated points.
        \State Compute the upper hull $\mathcal H_i^\theta$ and the slopes $\rho^{H,\theta}_{i,k}$ and lengths $\Delta c^{H,\theta}_{i,k}$.
    \EndFor
    \State Solve the LP relaxation of \eqref{eq:mckp_theta_clean} exactly via hull-greedy over $\rho^{H,\theta}_{i,k}$.
    \State Recover a discrete solution $\bm x(\theta)$ via \Cref{subsec:rounding} (round-down; optionally round-up+repair; optionally complete).
    \State Compute $N(\bm x(\theta))=\sum_{i=1}^n \omega_i\,x_i(\theta)\,\hat g_i(x_i(\theta))$.
    \State Compute the exact robust certificate $Z(\bm x(\theta)):=\sum_i s_i(x_i(\theta))-\beta(\bm x(\theta),\Gamma)$ as in \eqref{eq:certify_alg}.
    \If{$Z(\bm x(\theta))<0$} \textbf{continue}. \EndIf
    \If{$N(\bm x(\theta))>\textsc{BestVal}$}
        \State $\textsc{BestVal}\leftarrow N(\bm x(\theta))$; $\textsc{BestSol}\leftarrow (\theta,\bm x(\theta))$.
    \EndIf
\EndFor
\If{$\textsc{BestSol}=\emptyset$}
    \State \textbf{Output:} \textsc{Infeasible} (no $\theta\in\mathcal B$ yields $C^\theta\ge 0$ and a feasible selection).
\Else
    \State \textbf{Output:} $\textsc{BestSol}$.
\EndIf
\end{algorithmic}
\end{algorithm}

\paragraph{Practical note on $\theta$ enumeration.}
The exact candidate set $\mathcal B$ can be as large as $|\mathcal B|\le nm+1$ and retains every distinct original deviation.  Restricting candidates to breakpoints induced by an option filter is a heuristic unless the filter is proved threshold-safe; it is not used to support exactness claims.  Equation~\eqref{eq:s_theta_def_clean} implies
\[
s^\theta_{ij}=
\begin{cases}
s_{ij}, & \theta\ge |t_{ij}|,\\
s_{ij}-|t_{ij}|+\theta, & \theta<|t_{ij}|,
\end{cases}
\]
so $s^\theta_{ij}$ is piecewise-affine in $\theta$ with breakpoints at $|t_{ij}|$. Consequently, a sweep may track coefficient changes incrementally while still visiting the same complete original set.  The implementation reconstructs the authoritative fixed-threshold data at each breakpoint before computing a bound.  Hull reuse is valid only when the complete cost--value point set is unchanged; otherwise the hull must be rebuilt.  Regardless of the search strategy, the certificate \eqref{eq:certify_alg} provides an exact final check of each returned decision.

% ============================================================
% INTEGRALITY GAP (BENCHMARKING JUSTIFICATION)
% ============================================================
\subsection{Integrality gap and its asymptotic irrelevance}\label{subsec:igap}

For benchmarking, we compare discrete solutions against the LP relaxation of
\eqref{eq:mckp_theta_clean}.  At fixed $\theta$, the LP value is an upper bound
on the integer value, and the difference is controlled by a \emph{single} item.

\medskip
\noindent\textbf{Setup (fixed $\theta$).}
Fix $\theta\in\mathcal B$ with $C^\theta\ge 0$. For each item $i$, let
\[
\mathcal P_i^\theta:=\{(c^\theta_{ij},v_{ij}) : x_{i,j}\in\mathcal X_i^{\mathrm{ad}}\}\subset\mathbb R^2,
\]
and let $\mathcal H_i^\theta=\{(c^{H,\theta}_{i,1},v^{H,\theta}_{i,1}),\dots,(c^{H,\theta}_{i,p_i},v^{H,\theta}_{i,p_i})\}$
be the upper-hull vertex sequence of $\operatorname{conv}(\mathcal P_i^\theta)$, ordered by strictly increasing cost
$c^{H,\theta}_{i,1}<\cdots<c^{H,\theta}_{i,p_i}$.
Define the maximal adjacent hull jump
\[
\Delta V_{\max}^\theta
:=\max_{i}\ \max_{k=1,\dots,p_i-1}\big(v^{H,\theta}_{i,k+1}-v^{H,\theta}_{i,k}\big)\ \ge\ 0,
\]
with the convention $\max_{k\in\emptyset}(\cdot)=0$ when $p_i=1$.

\begin{proposition}[Additive integrality-gap bound]\label{prop:igap_additive_theta_clean}
For every fixed $\theta\in\mathcal B$ with $C^\theta\ge 0$,
\[
0\ \le\ \mathrm{OPT}_{\mathrm{LP}}(\theta)-\mathrm{OPT}_{\mathrm{IP}}(\theta)\ \le\ \Delta V_{\max}^\theta .
\]
\end{proposition}

\begin{proof}
The lower bound is immediate since the LP relaxation enlarges the feasible set. For the upper bound, by
\Cref{thm:master_greedy}(i)--(iii) there exists an optimal LP solution in which at most one item $i^\star$ is fractional, and if fractional
it mixes between two \emph{adjacent} hull vertices $(c^{H,\theta}_{i^\star,k},v^{H,\theta}_{i^\star,k})$ and
$(c^{H,\theta}_{i^\star,k+1},v^{H,\theta}_{i^\star,k+1})$ for some $k$. If no item is fractional, the LP solution is integral and the gap
is $0$.

Otherwise, the fractional point has the form
\[
(c_{i^\star},v_{i^\star})
=(1-\alpha)(c^{H,\theta}_{i^\star,k},v^{H,\theta}_{i^\star,k})+\alpha(c^{H,\theta}_{i^\star,k+1},v^{H,\theta}_{i^\star,k+1}),
\qquad \alpha\in(0,1).
\]
Round $i^\star$ down to the cheaper endpoint $(c^{H,\theta}_{i^\star,k},v^{H,\theta}_{i^\star,k})$ and keep all other items unchanged.
Total knapsack cost weakly decreases, hence feasibility is preserved. The objective decreases by at most the adjacent jump
$v^{H,\theta}_{i^\star,k+1}-v^{H,\theta}_{i^\star,k}$, because $v_{i^\star}$ lies between the endpoint values. Thus
\[
\mathrm{OPT}_{\mathrm{LP}}(\theta)-\mathrm{OPT}_{\mathrm{IP}}(\theta)
\ \le\ v^{H,\theta}_{i^\star,k+1}-v^{H,\theta}_{i^\star,k}
\ \le\ \Delta V_{\max}^\theta,
\]
as claimed.
\end{proof}

\begin{corollary}[Relative gap decays as \(O(1/n)\)]\label{cor:igap_relative_theta_clean}
Consider a sequence of instances indexed by the number of items $n$, with threshold set $\mathcal B_n$. Assume there exist constants $\overline V<\infty$ and $c>0$, independent of $n$, such that for every instance in the sequence and every $\theta\in\mathcal B_n$,
\[
\Delta V_{\max}^\theta\le \overline V,
\qquad
\mathrm{OPT}_{\mathrm{LP},n}(\theta)\ge c\,n.
\]
Then, uniformly over $\theta\in\mathcal B_n$,
\[
0\ \le\
\frac{\mathrm{OPT}_{\mathrm{LP},n}(\theta)-\mathrm{OPT}_{\mathrm{IP},n}(\theta)}{\mathrm{OPT}_{\mathrm{LP},n}(\theta)}
\ \le\
\frac{\overline V}{c\,n}
\ =\ O\!\left(\frac{1}{n}\right).
\]
\end{corollary}

\begin{proof}
For each $n$ and $\theta\in\mathcal B_n$, \Cref{prop:igap_additive_theta_clean} gives
$\mathrm{OPT}_{\mathrm{LP},n}(\theta)-\mathrm{OPT}_{\mathrm{IP},n}(\theta)\le \Delta V_{\max}^\theta\le \overline V$.
Divide by $\mathrm{OPT}_{\mathrm{LP},n}(\theta)\ge c\,n>0$. Because $\overline V$ and $c$ do not depend on $n$ or $\theta$, the bound is uniform over the threshold family.
\end{proof}

\paragraph{When do \texorpdfstring{$\overline V$}{Vbar} and $c$ exist? (Sufficient conditions)}
The assumptions $\Delta V_{\max}^\theta\le \overline V$ and $\mathrm{OPT}_{\mathrm{LP}}(\theta)\ge c\,n$ express two standard regularities:
(i) per-item value scales do not blow up, and (ii) the portfolio has nondegenerate average value. A convenient sufficient set is:

\medskip
\noindent\textbf{(A) Uniform bounded values.}
If there exists $M_v<\infty$ such that $0\le v_{ij}\le M_v$ for all items $i$, all admissible options $j\in J_i^{\mathrm{ad}}$, and all
$\theta\in\mathcal B$, then $\Delta V_{\max}^\theta\le M_v$ for all $\theta$ (take $\overline V:=M_v$), since hull vertices are menu points
and differences of two numbers in $[0,M_v]$ are at most $M_v$.
A concrete sufficient condition for such an $M_v$ is
$\sup_i\omega_i<\infty$, $\sup_{i,j}x_{i,j}<\infty$, and $\sup_{i,j}\hat g_i(x_{i,j})<\infty$, because then
$v_{ij}=\omega_i x_{i,j}\hat g_i(x_{i,j})$ is uniformly bounded.

\medskip
\noindent\textbf{(B) Uniform positive value in a feasible baseline.}
If there exists $m_v>0$ and, for every problem size and fixed $\theta$, a \emph{jointly feasible} selection $j_\theta(i)$ satisfying
$v_{i,j_\theta(i)}\ge m_v$ for every item, then that integer solution has value at least $m_v n$, hence
$\mathrm{OPT}_{\mathrm{LP}}(\theta)\ge \mathrm{OPT}_{\mathrm{IP}}(\theta)\ge m_v n$ (take $c:=m_v$).
In pricing applications this is typically ensured by assuming each item has a ``safe'' admissible baseline price that preserves the margin
constraint with slack, and exposures/demand at that baseline are bounded away from zero.

\paragraph{Practical implication.}
Under (A)--(B), the absolute LP--IP gap is controlled by a single-item jump while $\mathrm{OPT}_{\mathrm{LP}}(\theta)=\Omega(n)$, so the
relative gap vanishes along any sequence of instances satisfying those assumptions.  This is a conditional asymptotic statement, not a
claim that every large portfolio has a tight LP relaxation.

% ============================================================
% EXACT FULL-BREAKPOINT CERTIFICATION
% ============================================================
\section{Exact Full-Breakpoint Certification}\label{sec:exact}

HullRound supplies a robust-feasible decision and an additive fixed-$\theta$
certificate, but it need not solve the integer MCKP exactly.  When an exact
robust optimum or an anytime global gap is required, the same decomposition and
hull-greedy LP geometry can instead be used inside branch-and-bound.  The
distinction between LP geometry and integer search is essential throughout this
section.

\subsection{Fixed-\texorpdfstring{$\theta$}{theta} exact search}

Fix $\theta\in\mathcal B$ with $C^\theta\ge0$.  A branch-and-bound node fixes
options for a subset $F$ of the items and leaves $U:=\{1,\ldots,n\}\setminus F$
free.  If the fixed items use cost $C_F$ and contribute value $V_F$, the
remaining capacity is $C^\theta-C_F$.  The hull-greedy optimum of the residual
LP relaxation over $U$, plus $V_F$, is an upper bound on every integer
completion of the node.  Thus the fixed-$\theta$ LP routine from
\Cref{subsec:hull_theta} is also a structure-exploiting node-bound oracle.

Upper-hull filtering may be used to compute this LP bound, but it is not a
general integer-preprocessing rule.  An original option may be removed from
the exact search only if another option of the same item has weakly smaller
fixed-$\theta$ cost and weakly larger value, with at least one strict
inequality.  Non-dominated options below the upper hull remain searchable.

\begin{theorem}[Fixed-$\theta$ branch-and-bound exactness]\label{thm:fixed_theta_bnb}
Fix $\theta\in\mathcal B$.  Suppose branching covers every original option
remaining after the integer-safe within-item dominance rule, and suppose a node
is pruned only when it is infeasible or when its valid residual LP upper bound
is no larger than the incumbent value.  If every live node is eventually
fathomed, the returned incumbent is optimal for the fixed-$\theta$ MCKP.
\end{theorem}

\begin{proof}
The within-item dominance rule preserves at least one optimal completion,
because a removed option can be replaced without increasing cost or decreasing
value.  Branching partitions all remaining integer completions.  Infeasible
nodes contain no completion, while the residual LP bound dominates every
integer completion of a feasible node.  Consequently, a bound-pruned node
cannot improve the incumbent.  Once all nodes are fathomed, no unevaluated
integer completion can improve it.
\end{proof}

Every incumbent produced by the floating-point implementation is checked by
the original sorted-$\Gamma$ certificate in \eqref{eq:certify_alg}.  For the
feasibility decision, the binary64 input coefficients are converted to their
exact rational values and the certificate sign is tested without a feasibility
tolerance; LP feasibility, bound comparison, and objective comparison retain
their separately reported numerical tolerances.

\subsection{Global exactness and anytime gap}\label{subsec:global_exact}

Let $Z^{\mathrm{IP}}(\theta)$ denote the fixed-$\theta$ integer optimum.  By
\Cref{prop:theta_candidates_discrete_clean}, the robust optimum is
\[
Z^{\mathrm{rob}}=\max_{\theta\in\mathcal B}Z^{\mathrm{IP}}(\theta).
\]
The global algorithm initializes a valid record $U_\theta\ge
Z^{\mathrm{IP}}(\theta)$ for every original threshold before reporting a
finite gap.  A candidate with $C^\theta<0$ or an infeasible fixed-$\theta$ LP
receives $U_\theta=-\infty$; otherwise its root LP value is a valid initial
record.  As search proceeds, a record may be replaced by a proven optimum or
by the best remaining branch-and-bound upper bound.  A threshold is pruned only
when its valid record cannot improve the robust-feasible global incumbent
$Z^{\mathrm{inc}}$.

\begin{theorem}[Global full-breakpoint certificate]\label{thm:global_exact}
Construct $\mathcal B=\{0\}\cup\{|t_{ij}|\}$ from all original admissible
options.  If every unpruned fixed-$\theta$ search is completed under the
conditions of \Cref{thm:fixed_theta_bnb}, the best certified incumbent is a
globally optimal $\Gamma$-robust solution.  Under time or node limits, if every
$\theta\in\mathcal B$ has a valid upper-bound record, then
\[
 Z^{\mathrm{inc}}\ \le\ Z^{\mathrm{rob}}\ \le\
 U^{\mathrm{global}}:=\max_{\theta\in\mathcal B}U_\theta.
\]
Hence $U^{\mathrm{global}}-Z^{\mathrm{inc}}$ is a valid absolute anytime gap.
If even one original threshold lacks a valid record, no finite global gap is
certified.
\end{theorem}

\begin{proof}
The finite-breakpoint reduction covers every robust-feasible discrete
selection.  Completed fixed-$\theta$ searches are exact by
\Cref{thm:fixed_theta_bnb}, and a threshold pruned by a valid LP record cannot
improve the incumbent.  Under limits, the incumbent is a lower bound because
it passes the original robust certificate.  Each $U_\theta$ bounds its
fixed-threshold integer optimum, so their maximum bounds the maximum over the
complete threshold family.  The last conclusion follows because an unrecorded
threshold may contain an unbounded, relative to the recorded family, improving
solution.
\end{proof}

\subsection{Guarded parametric sweep}\label{subsec:sweep}

The guarded implementation traverses the same sorted full-breakpoint set and
tracks $s^\theta_{ij}$ through its piecewise-affine identity.  At each
breakpoint it independently reconstructs the authoritative transformed data,
including the exact capacity representation used by the certifying solver.
An item hull is reused only when its complete transformed
cost--value point multiset is unchanged; otherwise it is rebuilt from the
original admissible options.

\begin{proposition}[Validity of the guarded full-breakpoint sweep]\label{prop:sweep_exact}
In exact arithmetic, a sweep that visits the complete set $\mathcal B$,
constructs the same transformed fixed-$\theta$ data as independent
recomputation, and applies the guarded hull-reuse rule produces fixed-threshold
LP bounds and exact-search certificates with the same validity status as
independent enumeration.
\end{proposition}

\begin{proof}
Both procedures visit the same finite candidates.  The piecewise-affine update
is algebraically identical to \eqref{eq:s_theta_def_clean}.  Reusing an
unchanged point set preserves its hull; rebuilding a changed point set recovers
the independent hull.  Exact branch-and-bound continues to branch over the
integer-safe original options and uses the hull only as an LP-bound object, so
the arguments of \Cref{thm:fixed_theta_bnb,thm:global_exact} are unchanged.
\end{proof}

This proposition is a correctness statement, not a worst-case runtime
dominance claim.  The implementation preserves every binary64-distinct
original deviation plus zero, validates sampled swept states against
independent reconstruction, and fails closed rather than tolerance-clustering
threshold candidates.

% ============================================================
% SECTION: NUMERICAL RESULTS
% ============================================================
% Generated by scripts/run_paper_a_release.py; do not edit.
\newcommand{\AuditGlobalOracleCases}{1000}
\newcommand{\AuditBreakpointChecks}{62039}
\newcommand{\AuditLPChecks}{664}
\newcommand{\ExactComparisonRows}{45}
\newcommand{\MaxSolverDifference}{1.41e-09}
\newcommand{\MedianHullRoundGapPct}{0.00787\%}
\newcommand{\MaxHullRoundGapPct}{0.108\%}
\newcommand{\CertificateRows}{45}
\newcommand{\MaxLossBoundRatio}{0.535}
\newcommand{\MaxLPDifference}{2.33e-10}
\newcommand{\ScaleRows}{24}
\newcommand{\LargestN}{200}
\newcommand{\LargestM}{16}
\newcommand{\MedianHullFractionPct}{100\%}

\section{Numerical Results}\label{sec:numerics}

The computational study has three separate purposes: to challenge
the implementation against independent oracles, to evaluate the one-item
rounding certificate on controlled pricing instances, and to describe observed
runtime over the complete original breakpoint family.

\subsection{Protocol and reproducibility}

All results in this section come from one immutable dated evidence directory
generated by the repository's mathematical-audit and Paper-A-release drivers.
The directory records the git
commit, software environment, seeds, complete CSV rows, generated manuscript
macros, and SHA-256 manifest.  Candidate construction uses zero and every
binary64-distinct original deviation; no tolerance clustering is permitted.
Branch-and-bound uses a numerical comparison tolerance of $10^{-9}$, while
every reported incumbent is checked independently by directly sorting its
selected deviations and subtracting the largest $\Gamma$ values.  Thus the
exactness theorems remain exact-arithmetic statements and the experiment is a
binary64 implementation audit.

The independent mathematical audit contains \AuditGlobalOracleCases\ random
small instances split evenly between integer and widely scaled floating-point
coefficients.  For each instance, the full-breakpoint branch-and-bound result
is compared with exhaustive enumeration.  The audit also performs
\AuditBreakpointChecks\ selection-level checks of the finite-breakpoint
identity and \AuditLPChecks\ comparisons of the hull-greedy relaxation with an
independently formulated HiGHS LP.  It includes the permanent regression with
$100$ groups, $\Gamma=1$, and two distinct deviations $1$ and
$1+9\times10^{-11}$; deleting the latter breakpoint makes the constructed
instance infeasible.  All audit gates pass in the released report.

For the controlled pricing experiments, reference prices, exposures, demand
scales, elasticities, and menu perturbations are sampled independently from the
fully specified generator shipped with the repository.  Demand is decreasing
and piecewise linear on each finite menu, uncertainty is $10\%$ of nominal
demand, the fairness band is $\pm10\%$, and the margin target is calibrated so
the baseline portfolio retains $2\%$ nominal slack.

\subsection{Exact comparisons and rounding certificates}

We first compare the full-breakpoint branch-and-bound implementation with an
independent HiGHS mixed-integer formulation on \ExactComparisonRows\ instances:
five seeds, $n\in\{12,20,30\}$, $m=8$, and the distinct budgets among
$\{0,\lfloor\sqrt n\rfloor,\lfloor0.1n\rfloor\}$.  Both methods certify every
row, and their largest absolute objective difference is
\MaxSolverDifference.  HullRound's median and maximum relative gaps to those
certified optima are \MedianHullRoundGapPct\ and \MaxHullRoundGapPct,
respectively.

The certificate study uses \CertificateRows\ instances from five independent
seeds, $n\in\{30,60,120\}$, $m=10$, and the same three budget regimes.  Every
returned decision passes the direct sorted-$\Gamma$ robust check.  Every
round-down loss satisfies
$L_{\mathrm{rd}}\le\Delta V_{\max}^{\theta}$, with maximum observed ratio
\MaxLossBoundRatio.  At the selected thresholds, the largest absolute
difference between the hull-greedy LP value and the independent HiGHS LP value
is \MaxLPDifference.

\subsection{Controlled scalability observations}

The timing study contains \ScaleRows\ independently generated instances with
$n\in\{30,60,120,200\}$, $m\in\{8,16\}$, three seeds, and
$\Gamma=\lfloor\sqrt n\rfloor$.  It evaluates the complete original
breakpoint family and directly certifies every returned solution.  The largest
reported configuration is $n=\LargestN$, $m=\LargestM$.  The median fraction of
raw menu points retained as hull vertices is \MedianHullFractionPct, showing
that compression is instance dependent and can be weak for these economically
structured menus.  Wall-clock measurements are single-process observations on
the recorded machine and are intended only to indicate the scale of the public
implementation.

\begin{figure}[t]
\centering
\includegraphics[width=\textwidth]{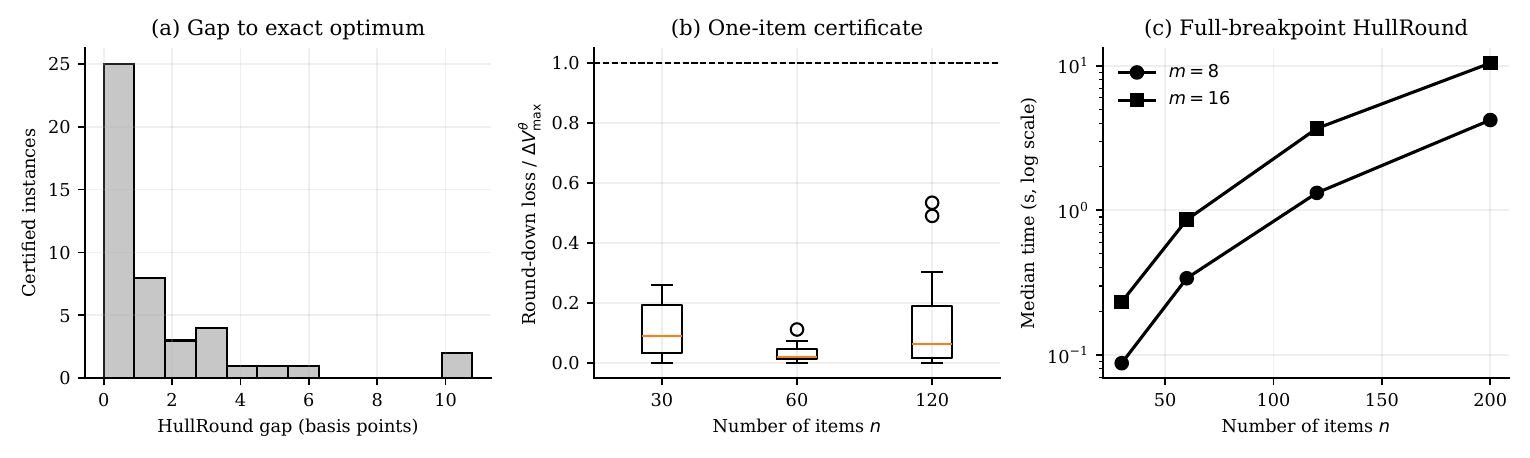}
\caption{Controlled evidence release. (a) HullRound gap to the independently
certified exact optimum. (b) Realized round-down loss divided by the one-item
bound; the dashed line is the theoretical ceiling. (c) Median wall-clock time
for complete-breakpoint HullRound.}
\label{fig:release_evidence}
\end{figure}

\paragraph{Limitations.}
The study uses controlled separable own-price demand rather than observed
transactions, does not model substitution or correlation, and is intentionally
bounded at moderate portfolio sizes so that others can reproduce it.  It
supports implementation correctness and illustrates the fixed-threshold
certificate.  It does not establish empirical superiority over general-purpose
solvers, estimate causal demand, or validate any particular commercial pricing
policy.

% CONCLUSION
% ============================================================
\section{Conclusion}\label{sec:conclusion}

We have studied discrete pricing under demand uncertainty through the lens of robust
optimization, reducing the problem to a multiple-choice knapsack problem with a
$\Gamma$-budget uncertainty set in its ratio-margin constraint.  The paper's
contribution is the certifying chain as a whole: the exact pricing-to-MCKP
transformation, the specialization to the complete original threshold family,
a robust-feasible one-item rounding certificate at each fixed threshold, and a
full-family branch-and-bound layer with valid global-gap accounting.  The
finite-family representation and MCKP hull geometry are classical ingredients;
their use in isolation is not claimed as novel.

The main guarantees are precise.  Round-down loses at most one
adjacent hull-value jump for a fixed threshold.  A relative $O(1/n)$ statement
requires uniformly bounded jumps and linear objective growth.  Global
optimality requires complete fixed-threshold search, while a finite anytime gap
requires a valid upper-bound record for every original threshold.  In the
binary64 implementation, every distinct original deviation is retained and
each incumbent is checked using the original sorted-$\Gamma$ expression.

The dated controlled evidence release passes exhaustive, independent-LP, and
independent-MILP comparisons and illustrates the certificate and computational
scale on synthetic pricing menus.

This endpoint also fixes the boundary with the subsequent paper.  The present
paper exposes and certifies the complete family of fixed-threshold MCKPs.  It
does not construct one bound that acts on many threshold intervals
simultaneously.  The subsequent paper starts from the family established here
and develops that separate simultaneous-envelope question, with different
theorems, algorithms, and evidence.  Extensions to correlated or substitutional
demand and validation on observed transaction data remain separate research
directions.

\section*{Data and Code Availability}
The solver, tests, immutable evidence release, and scripts used to generate
every reported numerical artifact are publicly available on GitHub:
\ifblind
\begin{center}
\emph{Anonymous artifact link supplied with the review submission.}
\end{center}
\else
\begin{center}
\url{https://github.com/eric939/certifying-robust-pricing-mckp}
\end{center}

\paragraph{Acknowledgments.} The author thanks Prof. Dr. Angela Kunoth (Department of Mathematics, University of Cologne) for her guidance and support. The problem was posed by Dr. Christian Mollet (Hannoversche Versicherung, Hannover, Germany), and the author also thanks him and his colleagues for several helpful discussions.
\fi

\paragraph{Generative-AI disclosure.} OpenAI ChatGPT and Codex assisted with language editing, code generation, and computational and mathematical auditing. The author independently reviewed and verified the manuscript text, proofs, code, and reported results and assumes full responsibility for the work.

% ---------- References ----------

% ---------- Appendix ----------
\appendix
\section{Knapsack Problem}\label{app:background}

\paragraph{Classical 0--1 Knapsack.}
Given items $i=1,\dots,n$ with value $v_i\ge 0$ and cost $c_i\ge 0$, and a capacity $C$, the 0--1 knapsack problem is
\begin{align}
\max \quad & \sum_{i=1}^n v_i\,y_i \\
\text{s.t.}\quad & \sum_{i=1}^n c_i\,y_i \le C, \\
& y_i \in\{0,1\}\quad (i=1,\dots,n).
\end{align}

\paragraph{Multiple-Choice Knapsack (MCKP).}
Early formulations and algorithms include \citet{Nauss1978}, \citet{Ibaraki1978}, and \citet{Sinha1979}. Items are partitioned into $n$ groups; from each group exactly one option must be chosen.
Let group $i$ have discrete options (price levels) indexed by $j$, with value $v_{ij}$ and cost $c_{ij}$. With capacity $C$, the MCKP is
\begin{align}
    \max \quad & \sum_{i=1}^n \sum_j v_{ij}\,z_{ij} \label{eq:mckp_obj}\\
    \text{s.t.}\quad & \sum_{i=1}^n \sum_j c_{ij}\,z_{ij} \le C, \label{eq:mckp_cap}\\
    & \sum_j z_{ij}=1 \quad (i=1,\dots,n), \label{eq:mckp_choice}\\
    & z_{ij}\in\{0,1\}\quad (\forall i,j). \label{eq:mckp_bin}
\end{align}

\section{Proof of the Main Structural Theorem, Steps 1--3}\label{app:thm2_steps}
We provide the first three steps of the proof of \Cref{thm:master_greedy} for completeness. These arguments are standard; we include them because our notation differs from prior sources.

\textbf{Step 1: Convex-hull reformulation.}
Fix an item $i$. For any feasible vector $(z_{ij})_{j=1}^m$ with $\sum_j z_{ij}=1$ and $z_{ij}\ge 0$, define
\[
(c_i,v_i):=\Big(\sum_{j=1}^m c_{ij}z_{ij},\ \sum_{j=1}^m v_{ij}z_{ij}\Big).
\]
Then $(c_i,v_i)\in \operatorname{conv}(\mathcal P_i)$ by definition of convex hull, since it is a convex combination of points in $\mathcal P_i$.
Conversely, every point in $\operatorname{conv}(\mathcal P_i)$ admits such a representation as a convex combination of points in $\mathcal P_i$.
Therefore, \eqref{eq:lp_mckp_master} is equivalent to choosing $(c_i,v_i)\in \operatorname{conv}(\mathcal P_i)$ for all $i$ to maximize
$\sum_i v_i$ subject to $\sum_i c_i\le C$.

\textbf{Step 2: Restriction to the upper hull (proof of (i)).}
Fix $i$ and a feasible solution with induced $(c_i,v_i)\in\operatorname{conv}(\mathcal P_i)$.
Consider the set
\[
U_i(c):=\sup\{v : (c,v)\in\operatorname{conv}(\mathcal P_i)\}.
\]
Since $\operatorname{conv}(\mathcal P_i)$ is a compact polytope, the supremum is attained and $U_i(c)$ describes the upper hull
of $\operatorname{conv}(\mathcal P_i)$.
Let $v_i':=U_i(c_i)$. Then $(c_i,v_i')\in\operatorname{conv}(\mathcal P_i)$, and $v_i'\ge v_i$ by construction.
Replacing $(c_i,v_i)$ by $(c_i,v_i')$ for each $i$ preserves feasibility of $\sum_i c_i\le C$ and weakly increases the objective.
Hence there exists an optimal solution with each $(c_i,v_i)$ on the upper hull.

The upper hull of a polytope in $\mathbb R^2$ is a polyline connecting consecutive vertices of $\mathcal H_i$.
Thus, if $(c_i,v_i)$ lies on the hull and $c_i\in[c^H_{i,k},c^H_{i,k+1}]$ for some $k$, then
$(c_i,v_i)$ is a convex combination of the two adjacent vertices
$(c^H_{i,k},v^H_{i,k})$ and $(c^H_{i,k+1},v^H_{i,k+1})$. This proves (i).

\textbf{Step 3: Segment-length parameterization (proof of (ii)).}
Fix an item $i$ and let the upper-hull vertices of $\operatorname{conv}(\mathcal P_i)$ be
\[
\mathcal H_i=\{(c^H_{i,1},v^H_{i,1}),\dots,(c^H_{i,p_i},v^H_{i,p_i})\},
\qquad c^H_{i,1}<\cdots<c^H_{i,p_i}.
\]
For $k=1,\dots,p_i-1$ define
\[
\Delta c^H_{i,k}:=c^H_{i,k+1}-c^H_{i,k}>0,
\qquad
\rho^H_{i,k}:=\frac{v^H_{i,k+1}-v^H_{i,k}}{\Delta c^H_{i,k}}.
\]
Every point $(c_i,v_i)$ on the upper hull lies on a unique segment between adjacent vertices, so there exist
$k^\star\in\{1,\dots,p_i-1\}$ and $\alpha\in[0,1]$ such that
\[
(c_i,v_i)=(1-\alpha)(c^H_{i,k^\star},v^H_{i,k^\star})+\alpha(c^H_{i,k^\star+1},v^H_{i,k^\star+1}).
\]
Define segment-length variables $\ell_{i,k}$ by
\[
\ell_{i,k}:=
\begin{cases}
\Delta c^H_{i,k}, & k<k^\star,\\
\alpha\,\Delta c^H_{i,k^\star}, & k=k^\star,\\
0, & k>k^\star,
\end{cases}
\qquad\text{so that}\qquad 0\le \ell_{i,k}\le \Delta c^H_{i,k}.
\]
Then the identities
\[
c_i=c^H_{i,1}+\sum_{k=1}^{p_i-1}\ell_{i,k},
\qquad
v_i=v^H_{i,1}+\sum_{k=1}^{p_i-1}\rho^H_{i,k}\,\ell_{i,k}
\]
hold, since on each segment $k$ the increase in $v$ equals $\rho^H_{i,k}$ times the increase in $c$.
Summing these expressions over $i$ and substituting into the convex-hull formulation from Step~1 yields
\[
\max \sum_{i=1}^n \Big(v^H_{i,1}+\sum_{k=1}^{p_i-1}\rho^H_{i,k}\ell_{i,k}\Big)
\quad \text{s.t.}\quad
\sum_{i=1}^n \Big(c^H_{i,1}+\sum_{k=1}^{p_i-1}\ell_{i,k}\Big)\le C,\quad
0\le \ell_{i,k}\le \Delta c^H_{i,k}.
\]
Moving $\sum_i c^H_{i,1}$ to the right-hand side gives \eqref{eq:cont_knap_master}.  The displayed prefix construction maps every hull
solution into that bounded continuous knapsack.  Conversely, concavity makes the slopes within an item nonincreasing, so a density-ordered
optimal solution of \eqref{eq:cont_knap_master} saturates segment $k$ before using segment $k+1$ (ties may be reordered this way without
changing value).  Such an optimum has the required prefix form and maps back to the hull.  Therefore the two problems have the same optimal
value, proving (ii).

\end{document}